\documentclass{article}
\usepackage{amsmath, amsthm, amssymb,color}
\usepackage{color}
\usepackage[dvips]{graphicx}
\usepackage{enumerate}
\def\authorsPS{}
\newcommand\authPS[8]{\ifnum\authPSc<1\def\authPSc{1}\else, \fi{\large\sffamily #1 #2}%
\edef\authorsPS{\authorsPS \par\vskip 1mm\noindent #1 #2:\hskip 5mm #3, #4, #5, #6, #7, #8}}

\newtheorem{theorem}{Theorem}
\newtheorem{lemma}[theorem]{Lemma}

\theoremstyle{definition}

\theoremstyle{remark}

\def\authPSc{0}
\newcommand{\beq}[1]{
\begin{equation}\label{#1}}
\newcommand{\eeq}{\end{equation}}
\newcommand{\req}[1]{{\rm(\ref{#1})}}
\newcommand{\hten}{{\mathfrak H}}

\newcommand{\Nt}{\lfloor nt \rfloor}
\newcommand{\Ns}{\lfloor ns \rfloor}
\newcommand{\DXj}{\Delta X_{\frac jn}}
\numberwithin{equation}{section}

\begin{document}



\title{Decomposition and limit theorems for a class of self-similar Gaussian processes}

\author{Daniel Harnett\thanks{
Department of Mathematical Sciences, University of Wisconsin Stevens Point \newline Stevens Point, Wisconsin 54481, dharnett@uwsp.edu}
\, and David Nualart\thanks{  
Department of Mathematics, University of Kansas\newline  
\; 405 Snow Hall, Lawrence, Kansas 66045-2142, nualart@ku.edu \newline D. Nualart is supported by NSF grant DMS1512891 and the ARO grant FED0070445 \newline
 {\bf Keywords}:  fractional Brownian motion, self-similar processes, stochastic heat equation, Hermite variations.  \newline
{\bf AMS 2010 Classification}:  60F05, 60G18, 60H07}  
}

\maketitle

\begin{abstract}
We introduce a new class of self-similar Gaussian stochastic processes, where the covariance is defined in terms of a fractional Brownian motion and another Gaussian process.  A special case is the solution in time to the fractional-colored stochastic heat equation described in Tudor (2013).  We prove that the process can be decomposed into a fractional Brownian motion (with a different parameter than the one that defines the covariance), and a Gaussian process first described in Lei and Nualart (2008).  The component processes can be expressed as stochastic integrals with respect to the Brownian sheet.  We then prove a central limit theorem about the Hermite variations of the process.
 \end{abstract}

\large
\section{Introduction}
 The purpose of this paper is to introduce a new class of Gaussian self-similar stochastic processes related to stochastic partial differential equations, and to establish a decomposition in law and  a central limit theorem for the Hermite variations of the increments of such processes. 
 
 Consider the $d$-dimensional stochastic  heat equation
 \begin{equation} \label{SHE}
\frac{\partial u}{\partial t} = \frac 12\Delta u + \dot{W},\; t\ge 0,\; x\in \mathbb{R}^d,
\end{equation} 
 with zero initial condition, where $\dot{W}$ is a zero mean Gaussian field with a covariance of the form
 \[ 
\mathbb{E}\left[\dot{W}^H(t,x)\dot{W}^H(s,y)\right] =  \gamma_0(t-s) \Lambda(x-y), \quad s,t \ge 0 ,\; x,y \in \mathbb{R}^d.
\]
We are interested in the the process  $U=\{ U_t, t\ge 0\}$, where  $U_t= u(t,0)$.

Suppose that    $\dot{W}$ is white in time, that is,  $\gamma_0 =\delta_0$ and the spatial covariance is the Riesz kernel, that is,
$\Lambda(x)=c_{d,\beta }  |x|^{-\beta}$, with $\beta<\min(d,2)$ and $c_{d,\beta} = \pi^{-d/2} 2^{\beta-d} 
\Gamma(\beta /2) / \Gamma((d-\beta)/2)$. Then $U$ has the covariance (see \cite{Tu_Xiao})
\begin{equation}  \label{cov1}
\mathbb{E} [ U_tU_s]=  D
 \left( (t+s)^{1-\frac \beta 2}  -|t-s|^{1-\frac  \beta 2} \right), \quad s,t \ge 0,
\end{equation}
for some constant  
\begin{equation}  \label{D}
D= (2\pi )^{-d}  (1- \beta /2)^{-1} \int_{\mathbb{R}^d} e^{-\frac{|\xi|^2}{2}}  \frac {d\xi}  {|\xi|^{d-\beta}}.
\end{equation}
Up to a constant, the covariance (\ref{cov1})  is the covariance of the {\it bifractional Brownian motion} with parameters $H= \frac 12$ and $K = 1-\frac \beta 2$.   We recall that, given constants $H\in (0,1)$ and $K\in (0,1)$, the bifractional Brownian motion
$B^{H,K} =\{ B^{H,K}_t, t\ge 0\}$,   introduced in 
\cite{HoVi},   is a centered Gaussian process with covariance
\[
R_{H,K}(s,t)= \frac 1{2^K}\left(( t^{2H} + s^{2H})^K - |t-s|^{2HK} \right),  \quad s,t \ge 0.
\]
When $K=1$, the process $B^H=B^{H,1}$ is simply the fractional Brownian motion (fBm)  with Hurst parameter $H\in (0,1)$, with covariance $R_H(s,t)= R_{H,1}(s,t)$. In \cite{Lei},  Lei and Nualart obtained the following  decomposition in law for the bifractional Brownian motion
\[
B^{H,K}= C_1 B^{HK} + C_2 Y^K_{t^{2H}},
\]
where $B^{HK}$ is a fBm   with Hurst parameter $HK$, the process $Y^K$ is given by    
\begin{equation}  \label{y}
Y^K_t = \int_0^\infty y^{-\frac{1+K} 2} (1-e^{-yt}) dW_y,
\end{equation}
  with $W=\{W_y, y\ge0 \}$
 a standard Brownian motion independent of $B^{H,K}$, and   $C_1, C_2$ are constants given by $C_1=2^{\frac{1-K} 2}$ and
$C_2=\sqrt{ \frac {2^{-K}}{\Gamma(1-K)}}$.  The process $Y^K$ has trajectories which are infinitely differentiable on $(0,\infty)$ and H\"older continuous of order $HK-\epsilon$ in any interval $[0,T]$ for any $\epsilon>0$. 
In particular, this leads to a decomposition in law of the process $U$ with covariance (\ref{cov1}) as the sum of a fractional Brownian motion with Hurst parameter $\frac 12-\frac \beta 4$ plus a regular process.  

The classical  one-dimensional space-time white noise can also be considered as an extension of the covariance (\ref{cov1}) if we take $\beta=1$. In this case the covariance corresponds, up to a constant, to that of a bifractional Brownian motion with  parameters $H=K=\frac 12$. 

The case where     the noise term $\dot{W}$ is a fractional Brownian motion with Hurst parameter $H\in(\frac 12, 1)$ in time and  a spatial covariance  given by the Riesz kernel, that is,
\[ 
\mathbb{E}\left[\dot{W}^H(t,x)\dot{W}^H(s,y)\right] = \alpha_H c_{d,\beta}  |s-t|^{2H-2}|x-y|^{-\beta},
\]
where $0< \beta < \min( d,2)$ and $\alpha_H =H(2H-1)$,  has  been considered  by  Tudor and Xiao in \cite{Tu_Xiao}.  In this case the corresponding process $U$ has the covariance
\begin{equation}  \label{cov2}
\mathbb{E}[ U_tU_s] = D\alpha_H \int_0^t \int_0^s  |u-v|^{2H-2}(t+s-u-v)^{-\gamma}dudv.
\end{equation}
where $D$ is given in (\ref{D}) and  $\gamma=\frac {d-\beta} 2$. This process is self-similar with parameter 
$H- \frac \gamma 2$ and it
  has been studied in a series of papers  \cite{Balan, Ouahhabi, TTV, Tudor, Tu_Xiao}. In particular, in 
\cite{Tu_Xiao}  it is proved that the process $U$ can be decomposed into the sum of a scaled fBm with parameter $H-\frac{\gamma}{2}$, and a Gaussian process $V$ with continuously differentiable trajectories. 
This decomposition is based on the  stochastic heat equation. As a consequence, one can derive the exact uniform and local moduli of continuity and Chung-type laws of  the iterated logarithm for this process.
  In  \cite{TTV}, assuming that $d=1,2$ or $3$, a central limit theorem is obtained for the
    renormalized quadratic variation
 \[ 
 V_n =n^{2H-\gamma -\frac 12}  \sum_{j=0}^{n-1}  \left\{(U_{ (j+1)T/n}-U_{jT/n} )^2 - \mathbb{E}\left[(U_{(j+1)t/n}-U_{jT/n})^2\right]\right\},
\] 
assuming  $\frac 12 < H < \frac 34$, extending well-known results for fBm  (see for example \cite[Theorem 7.4.1]{NoP11}). 
  
  The purpose of this paper is to   establish a decomposition in law, similar to that obtained by Lei and Nualart in \cite{Lei} for the bifractional Brownian motion,  and  a central limit theorem for  the Hermite variations of the increments, for a class of  self-similar processes  that includes the covariance  (\ref{cov2}).
 Consider a centered Gaussian process
 $\{X_t, t\ge 0\}$ with covariance
 \begin{equation} \label{eq1}
 R(s,t) =\mathbb{E} [X_sX_t] = \mathbb{E}  \left[ \left(\int_0^t Z_{t-r} dB^H_r \right)\left( \int_0^sZ_{s-r} dB^H_r   \right)\right],
\end{equation}
where  
\begin{itemize}
\item[(i)] $B^H=\{B^H_t, t\ge 0\}$ is a  fBm  with Hurst parameter $H\in (0,1)$.  
\item[(ii)] $Z=\{Z_t, t> 0\}$ is a zero-mean Gaussian process, independent of $B^H$, with covariance
 \begin{equation} \label{eq2}
\mathbb{E} [Z_sZ _t] =(s+t)^{-\gamma},
\end{equation}
where $0<\gamma <2H$.
\end{itemize}
In other words, $X$ is a Gaussian process with the same covariance as the process $\{\int_0^t Z_{t-r} dB^H_r, t\ge 0\}$, which is not Gaussian. 

\medskip
When $H\in (\frac 12, 1)$, the covariance (\ref{eq1})  coincides with (\ref{cov2}) with $D=1$. 
 However, we allow the range of parameters $0 < H < 1$ and $0 < \gamma < 2H$.  In other words, up to a constant, $X$ has the law of the solution in time of the stochastic heat equation (\ref{SHE}), when $H\in (0,1)$ and  $d\ge 1$ and $\beta=d-2\gamma$. 
 Also of interest is that $X$ can be constructed as a sum of stochastic integrals with respect to the Brownian sheet (see the proof of Theorem 1).

\medskip
\subsection{Decomposition of the process $X$}
Our first result is the following decomposition in law of the process $X$ as the sum of a fractional Brownian motion with Hurst parameter $\frac{\alpha}{2}= H-\frac{\gamma}{2}$ plus a process with regular trajectories.
\begin{theorem} \label{thm1}
The process $X$ has the same law as
$\{\sqrt{\kappa}   B^{ \frac \alpha 2}_t+ Y_t, t\ge 0\}$, 
where 
\begin{equation} \label{kappa}
 \kappa =  \frac{1}{\Gamma(\gamma)}\int_0^\infty \frac{z^{\gamma-1}}{1+z^2}~dz,
\end{equation}
 $B^{\frac \alpha 2}$ is a fBm with Hurst parameter $\alpha/2$, and $Y$  (up to a constant)  has the same law as the  process $Y^K$ defined in (\ref{y}), with $K=2\alpha +1$,
 that is, $Y$ is a centered Gaussian process with covariance given by
 \[ \mathbb{E}\left[ Y_t Y_s\right] = \lambda_1\int_0^\infty y^{-\alpha-1}(1 - e^{-yt})(1-e^{-ys})~dy,\]
 where
 \[\lambda_1 = \frac{4\pi}{\Gamma(\gamma)\Gamma(2H+1)\sin(\pi H)}\int_0^\infty \frac{\eta^{1-2H}}{1+\eta^2}d\eta.\]
\end{theorem}
The proof of this theorem is given in Section 3.

\subsection{Hermite variations of the process}
For each integer $q \ge 0$, the $q$th Hermite polynomial is given by
\[ H_q(x) = (-1)^q e^{\frac{x^2}{2}} \frac{d^q}{dx^q} e^{-\frac{x^2}{2}}.\]
See \cite[Section 1.4]{NoP11} for a discussion of properties of these polynomials.  In particular, it is well known that the family $\{\frac{1}{q!} H_q, q\ge 0\}$ constitutes an orthonormal basis of the space $L^2(\mathbb{R},\gamma)$, where $\gamma$ is the $N(0,1)$ measure.

Suppose $\{ Z_n, n \ge 1\}$ is a stationary, Gaussian sequence, where each $Z_n$ follows the $N(0,1)$ distribution with covariance function $\rho(k) = \mathbb{E}\left[ Z_n Z_{n+k}\right]$.  If $\sum_{k=1}^\infty |\rho(k)|^q < \infty$, it is well known that as $n$ tends to infinity, the Hermite variation
\beq{Breuer_Major_classic} V_n = \frac{1}{\sqrt n} \sum_{j=1}^n H_q(Z_j)\eeq
converges in distribution to a Gaussian random variable with mean zero and variance given by $\sigma^2 = \sum_{k=1}^\infty\rho(k)^q$.  This result was proved by Breuer and Major in \cite{BreuerMajor}.  In particular, if $B^H$ is a fBm, then the sequence $\{ Z_{j,n}, 0\le j \le n-1\}$ defined by
\[ Z_{j,n} = n^H\left(B^H_{\frac{j+1}{n}}-B^H_{\frac jn}\right)\]
is a stationary sequence with unit variance.   As a consequence,  $H < 1-\frac 1q,$   we have that
\[ \frac{1}{\sqrt n} \sum_{j=0}^{n-1} H_q\left(n^H\left(B^H_{\frac{j+1}{n}}-B^H_{\frac jn}\right)\right)\]
converges to a normal law with variance given by
\beq{Variance_fBm} \sigma_q^2 =  \frac{q!}{2^q} \sum_{m\in{\mathbb Z}} \left( |m+1|^{2H}-2|m|^{2H}+|m-1|^{2H}\right)^q.\eeq
See \cite{BreuerMajor} and Theorem 7.4.1 of \cite{NoP11}.

The above Breuer-Major theorem can not be applied to our process because $X$ is not necessarily stationary.  However, we have a comparable result.  

\begin{theorem}\label{thm2} Let $q \ge 2$ be an integer and fix a real $T>0$.  Suppose that $\alpha  < 2-\frac 1q$.   For $t\in [0,T]$, define,  
\[ F_n(t) = n^{-\frac{1}{2}}\sum_{j=0}^{\Nt -1}  H_q\left( \frac{\DXj}{\left\|\DXj \right\|_{L^2(\Omega)}}\right),
\]
where $H_q(x)$ denotes the $q$th Hermite polynomial.  Then as $n\to\infty$,  the stochastic process $\{F_n(t), t\in [0,T]\}$ converges in law in the Skorohod space $D([0,T])$,  to a scaled Brownian motion $\{\sigma B_t, t\in [0,T]\}$, where $\{B_t, t\in [0,T]\} $ is a standard Brownian motion and $\sigma=\sqrt{\sigma^2}$ is given by
\beq{sigma}
\sigma^2 =  
\frac{q!}{2^q}\sum_{m\in\mathbb{Z}}\left( |m+1|^\alpha - 2|m|^\alpha+|m-1|^\alpha\right)^q.\eeq

\end{theorem}
The proof of this theorem is given in Section 4.

\section{Preliminaries}
\subsection{Analysis on the Wiener space}
The reader may refer to \cite{NoP11, Nualart} for a detailed coverage of this topic. Let $Z = \{ Z(h), h\in\cal{H}\}$ be an {\em isonormal Gaussian process} on a probability space $( \Omega, {\cal F}, P )$,  indexed by a real separable Hilbert space $\cal{H}$.  This means that $Z$ is a family of Gaussian random variables such that ${\mathbb E}[ Z(h)] =0$ and ${\mathbb E}\left[Z(h)Z(g)\right] = \left< h,g\right>_{\cal{H}}$ for all $h,g \in\cal{H}$.
  
For integers $q \ge 1$, let ${\cal H}^{\otimes q}$ denote the $q$th tensor product of ${\cal H}$, and ${\cal H}^{\odot q}$ denote the subspace of symmetric elements of ${\cal H}^{\otimes q}$.

Let $\{ e_n, n\ge 1\}$ be a complete orthormal system in ${\cal H}$.  For elements $f, g \in {\cal H}^{\odot q}$ and $p\in\{0, \dots, q\}$, we define the $p$th-order contraction of $f$ and $g$ as that element of ${\cal H}^{\otimes 2(q-p)}$ given by
\beq{contract} f\otimes_p g = \sum_{i_1, \dots , i_p=1}^\infty \left< f, e_{i_1} \otimes \cdots\otimes e_{i_p}\right>_{{\cal H}^{\otimes p}} \otimes \left< g, e_{i_1} \otimes \cdots\otimes e_{i_p}\right>_{{\cal H}^{\otimes p}}, \eeq
where $f\otimes_0 g = f\otimes g$. Note that, if $f,g \in {\cal H}^{\odot q}$, then  $f\otimes_q g = \left< f,g\right>_{{\cal H}^{\odot q}}$.  In particular, if $f,g$ are real-valued functions in ${\cal H}^{\otimes q} = L^2({\mathbb R}^2, {\cal B}^2, \mu)$ for a non-atomic measure $\mu$, then we have 
\beq{contract_integrl} f \otimes_1 g = \int_{\mathbb{R}} f(s, t_1) g(s, t_2)~\mu(ds).\eeq

Let ${\cal H}_q$ be the $q$th Wiener chaos of $Z$, that is, the closed linear subspace of $L^2(\Omega)$ generated by the random variables $\{ H_q(Z(h)), h \in {\cal H}, \|h \|_{\cal H} = 1 \}$, where $H_q(x)$ is the $q$th Hermite polynomial.  It can be shown (see \cite[Proposition 2.2.1]{NoP11}) that if $Z, Y \sim N(0,1)$ are jointly Gaussian, then
\beq{Herm_cov} {\mathbb E}\left[ H_p(Z) H_q(Y)\right] = \begin{cases}p!\left({\mathbb E}\left[ ZY\right]\right)^p & if\; p=q\\0&  {\rm otherwise}\end{cases}.\eeq
For $q \ge 1$, it is known that the map 
\beq{Hmap} I_q(h^{\otimes q}) = H_q(Z(h))\eeq
provides a linear isometry between ${\cal H}^{\odot q}$ (equipped with the modified norm $\sqrt{q!}\| \cdot\|_{{\cal H}^{\otimes q}}$) and ${\cal H}_q$, where $I_q(\cdot)$ is the generalized Wiener-It\^o stochastic integral (see \cite[Theorem 2.7.7]{NoP11}).  By convention, ${\cal H}_0 = \mathbb{R}$ and $I_0(x) = x$.  

We use the following integral multiplication theorem from \cite[Proposition 1.1.3]{Nualart}.  Suppose $f\in {\cal H}^{\odot p}$ and $g\in {\cal H}^{\odot q}$.  Then
\beq{int_mult}
I_p(f)I_q(g) = \sum_{r=0}^{p\wedge q} r! \binom pr \binom qr I_{p+q-2r} (f \widetilde{\otimes}_r g) ,\eeq
where $  f \widetilde{\otimes}_r g$ denotes the symmetrization of  $f \otimes_r g$.
For a product of more than two integrals, see Peccati and Taqqu \cite{PTaq}.

\medskip
\subsection{Stochastic integration and fBm}
We refer to the `time domain' and `spectral domain' representations of fBm.  The reader may refer to \cite{PipTaq, SamoTaq} for details.  Let $\cal E$ denote the set of real-valued step functions on $\mathbb{R}$.  Let $B^H$ denote fBm with Hurst parameter $H$.  For this case, we view $B^H$ as an isonormal Gaussian process on the Hilbert space $\hten$, which is the closure of $\cal E$ with respect to the inner product $\left< f, g\right>_\hten = \mathbb{E}\left[I(f) I(g)\right]$.  Consider also the inner product space
\[ \tilde{\Lambda}_H = \left\{ f: f\in L^2(\mathbb{R}), \int_{\mathbb{R}} |{\cal F}f(\xi)|^2|\xi|^{1-2H}d\xi < \infty\right\},\]
where ${\cal F}f = \int_{\mathbb{R}} f(x) e^{i\xi x}dx$ is the Fourier transform, and the inner product of $\tilde{\Lambda}_H$ is given by
\beq{spectral}
\left< f, g\right>_{\tilde{\Lambda}_H} = \frac{1}{C_H^2} \int_{\mathbb{R}} {\cal F}f(\xi)\overline{{\cal F}g(\xi)}|\xi|^{1-2H}d\xi,\eeq
where $C_H = \left(\frac{2\pi}{\Gamma(2H+1)\sin(\pi H)}\right)^{\frac 12}$.  It is known  (see \cite[Theorem 3.1]{PipTaq}) that the space $\tilde{\Lambda}_H$ is isometric to a subspace of $\hten$, and $\tilde{\Lambda}_H$ contains $\cal E$ as a dense subset.  This inner product \req{spectral} is known as the `spectral measure' of fBm.  In the case $H \in (\frac 12, 1)$, there is another isometry from the space
\[ \left| \Lambda_H\right| = \left\{ f:  \int_0^\infty\int_0^\infty |f(u)||f(v)| |u-v|^{2H-2}du~dv < \infty\right\}\]
to a subspace of $\hten$, where the inner product is defined as
\[ \left< f, g\right>_{|\Lambda_H|} = H(2H-1)\int_0^\infty \int_0^\infty f(u)g(v)|u-v|^{2H-2}du~dv,\]
see \cite{PipTaq} or \cite[Section 5.1]{Nualart}.

\section{Proof of Theorem \ref{thm1}}

 For any $\gamma >0$ and $\lambda >0$,  we can write
 \[
  \lambda^{-\gamma} = \frac{1}{\Gamma(\gamma)}\int_0^\infty y^{\gamma-1}e^{-\lambda y}dy,
  \]
where $\Gamma$ is the Gamma function defined by
$\Gamma(\gamma) = \int_0^\infty y^{\gamma-1}e^{-y}dy$.    As a consequence, the covariance (\ref{eq2}) can be written as
\begin{equation} \label{Zcov}
\mathbb{E} [Z_sZ_t] =  \frac{1}{\Gamma(\gamma)}\int_0^\infty y^{\gamma-1}e^{-(t+s) y}dy.
\end{equation}
Notice that this representation implies the covariance \req{eq2} is positive definite.  Taking first the expectation with respect to the process $Z$,  and using formula (\ref{Zcov}), we obtain
\begin{eqnarray*}
R(s,t) &=&   \frac{1}{\Gamma(\gamma)}\int_0^\infty  
\mathbb{E} \left[  \left( \int_0^t e^{yu} dB^H_u \right) \left( \int_0^t e^{yu} dB^H_u \right)  \right] y^{\gamma-1}e^{-(t+s) y}dy \\
&=& \frac{1}{\Gamma(\gamma)}\int_0^\infty  
 \left \langle    e^{yu}\mathbf{1}_{[0,t]}(u),  e^{yv}\mathbf{1}_{[0,s]}(v)   \right\rangle _\hten  y^{\gamma-1}e^{-(t+s) y}dy.
\end{eqnarray*}
Using the isometry between $\tilde{\Lambda}_H$ and a subspace of $\hten$ (see section 2.2), we can write
\begin{eqnarray*}
\left \langle    e^{yu}\mathbf{1}_{[0,t]}(u),  e^{yv}\mathbf{1}_{[0,s]}(v)   \right\rangle _\hten 
&=&  C_H^{-2} \int_{\mathbb{R}} |\xi|^{1-2H}({\cal F}{\mathbf{1}_{[0,t]}e^{y\cdot}})(\overline{{\cal F}{\mathbf{1}_{[0,s]}e^{y\cdot}}})~d\xi\\
&=& C_H^{-2} \int_{\mathbb{R}} \frac{|\xi|^{1-2H}}{y^2+\xi^2}\left( e^{yt+i\xi t}-1\right)\left(e^{ys - i\xi s}-1\right)~d\xi,
\end{eqnarray*} 
where $({\cal F}{\mathbf{1}_{[0,t]}e^{x\cdot}})$ denotes the Fourier transform and 
$C_H = \left(\frac{2\pi}{\Gamma(2H+1)\sin(\pi H)}\right)^{\frac 12}$.
  This allows us to write, making  the  change of variable $\xi = \eta y$,
\begin{eqnarray}  \notag
R(s,t) &=&  \frac 1{ \Gamma(\gamma) C_H^{2}}  \int_0^\infty\int_{\mathbb{R}} y^{\gamma-1}\frac{|\xi|^{1-2H}}{y^2+\xi^2}\left( e^{i\xi t}-e^{-yt}\right)\left(e^{-i\xi s}-e^{-ys}\right)~d\xi~dy\\
&=& \frac 1{ \Gamma(\gamma) C_H^{2}}   \int_0^\infty\int_{\mathbb{R}} y^{-\alpha-1} \frac{|\eta| ^{1-2H}}{1+\eta^2}\left( e^{i\eta yt}-e^{-yt}\right)\left( e^{-i\eta y s}-e^{-ys}\right)~d\eta~dy,  \label{SpecR}
\end{eqnarray} 
where   $\alpha = 2H-\gamma$. 
By Euler's identity,  adding and subtracting $1$ to compensate the singularity of  $y^{-\alpha-1}$ at the origin, we can write
\begin{equation}  \label{eq3}
e^{i \eta y t} - e^{-yt} = (\cos(\eta y t) -1 + i\sin(\eta y t)) + (1-e^{-yt}).
\end{equation}
Substituting (\ref{eq3}) into (\ref{SpecR}) and taking into account that the integral of the imaginary part vanishes because it is an odd function, we obtain
\begin{eqnarray*}
R(s,t)&=&  \frac 2{ \Gamma(\gamma) C_H^{2}}   \int_0^\infty\int_0^\infty  y^{-\alpha-1} \frac{\eta ^{1-2H}}{1+\eta^2}\Big((1-\cos(\eta y t)) (1-\cos(\eta y s)) \\
&&+\sin(\eta y t )\sin(\eta y s)+(\cos(\eta y s)-1)(1-e^{-yt})+ (\cos(\eta y t)-1)(1-e^{-ys})\\
&& +(1-e^{-yt})(1-e^{-ys})\Big)~d\eta~dy. \label{eq4}
\end{eqnarray*}

Let  $B^{(j)} = \{ B^{(j)}(\eta, t), \eta\ge 0, t\ge 0\}$, $j=1,2$ denote two independent Brownian sheets.  That is,  for $j=1,2$, $B^{(j)}$ is a continuous Gaussian field with mean zero and covariance given by
\[
\mathbb{E}\left[ B^{(j)}(\eta, t) B^{(j)}(\xi, s)\right] = \min(\eta , \xi)\times \min(t, s).
\] 
 We define the the following stochastic processes:
\begin{align}\label{UVY}
U_t = \frac  {\sqrt{2}}  {\sqrt{\Gamma(\gamma)} C_H} &\int_0^\infty \int_0^\infty y^{-\frac{\alpha}{2}- \frac 12} \sqrt{\frac{\eta^{1-2H}}{1+\eta^2}} \left( \cos(\eta y t)-1\right) B^{(1)}(d\eta, dy),\\
V_t =\frac  {\sqrt{2}}  {\sqrt{\Gamma(\gamma)} C_H} &\int_0^\infty \int_0^\infty y^{-\frac{\alpha}{2}- \frac 12} \sqrt{\frac{\eta^{1-2H}}{1+\eta^2}} \left( \sin(\eta y t)\right) B^{(2)}(d\eta, dy),\\
Y_t =\frac  {\sqrt{2}}  {\sqrt{\Gamma(\gamma)} C_H} &\int_0^\infty \int_0^\infty y^{-\frac{\alpha}{2}- \frac 12} \sqrt{\frac{\eta^{1-2H}}{1+\eta^2}} \left( 1-e^{-yt}\right) B^{(1)}(d\eta, dy),
\end{align}
where the integrals are  Wiener-It\^o integrals with respect to  the Brownian sheet.  We then define the stochastic process $X = \{ X_t, t\ge 0\}$ by $X_t = U_t + V_t + Y_t$, and we have $\mathbb{E}\left[ X_s X_t\right] = R(s,t)$ as given in \req{SpecR}.    These processes have the following properties:

\medskip
\noindent
(I) The process  $W_t=U_t+V_t $ is a fractional Brownian motion with Hurst parameter $\frac{\alpha}{2}$ scaled with the constant $\sqrt{\kappa}$. In fact, the covariance of this process is
\begin{eqnarray*}
\mathbb{E} [W_tW_s]&=& \frac 2{ \Gamma(\gamma) C_H^{2}}  \int_0^\infty \int_0^\infty y^{- \alpha -1}  \frac{\eta^{1-2H}}{1+\eta^2} \Big( (\cos(\eta y t)-1) (\cos(\eta y s)-1)\\
&& + \sin(\eta y t)\sin(\eta y s)\Big)  d\eta dy\\
&=& \frac 1{ \Gamma(\gamma) C_H^{2}}  \int_0^\infty \int_{\mathbb{R}} y^{\gamma -1}  \frac{|\xi|^{1-2H}}{y^2+\xi^2} (  e^{i\xi t}-1)(e^{-i\xi s }-1)  d\xi dy.
\end{eqnarray*} 
Integrating in the variable $y$ we finally obtain
\[
\mathbb{E} [W_tW_s]
= \frac {c_1}{ \Gamma(\gamma) C_H^{2}}  \int_{\mathbb{R}}  \frac{(  e^{i\xi t}-1)(e^{-i\xi s }-1)}{|\xi|^{\alpha+1}} d\xi,
\]
where $c_1= \int_0^\infty \frac {z^{\gamma-1}}{1+z^2} dz = \kappa \Gamma(\gamma)$.
Taking into account the 
 Fourier transform representation of fBm (see \cite[page 328]{SamoTaq}), this  
  implies $\kappa^{-\frac 12} W $ is a fractional Brownian motion with Hurst parameter $\frac{\alpha}{2}$.
  
\medskip
\noindent
(II)  The process $Y$ coincides, up to a constant,  with the process  $Y^K$ introduced in (\ref{y}) with  $K= 2\alpha +1$. In fact, the covariance of this process is given by
\beq{CovYY}
\mathbb{E} [Y_tY_s]= \frac {2c_2}{ \Gamma(\gamma) C_H^{2}}  \int_0^\infty y^{- \alpha -1}    ( 1-e^{-yt})( 1-e^{-ys})  dy,
\eeq
where
\[
c_2=   \int_0^\infty  \frac{\eta^{1-2H}}{1+\eta^2} d\eta.
\]

Notice that the process $X$ is self-similar with exponent $\frac{\alpha}{2}$.
 This concludes the proof of Theorem \ref{thm1}.

\section{Proof of Theorem \ref{thm2}}
\medskip
Along the proof,  the symbol $C$ denotes a generic, positive constant, which may change from line to line. The value of $C$ will depend on  parameters of the process and on $T$, but not on the increment width $n^{-1}$.

For integers $n\ge 1$, define a partition of $[0,\infty)$ composed of the intervals $\{[\frac{j}{n},\frac{j+1}{n}), j\ge 0\}$.  For the process $X$ and related processes $U, V, W, Y$ defined in Section 3, we introduce the notation
\[\Delta X_{\frac jn} = X_{\frac{j+1}{n}} - X_{\frac jn}\;\text{and }\, \Delta X_0 = X_{\frac 1n},\]
with corresponding notation for $U, V, W, Y$.  We start the proof of Theorem \ref{thm2} with two technical results about the components of the increments.

\subsection{Preliminary Lemmas}
\begin{lemma}\label{Lem4_1}Using above notation with integers $n \ge 2$ and $j,k\ge 0$, we have
\begin{enumerate}[(a)]
\item $\mathbb{E}\left[ \Delta W_{\frac jn}\Delta W_{\frac kn}  \right] = \frac {\kappa}2 n^{-\alpha} \left( |j-k-1|^\alpha - 2 |j-k|^\alpha+|j-k-1|^\alpha\right)$,  where $\kappa$ is defined in (\ref{kappa}). 
 
\item  For $j+k \ge 1$,\[\left| \mathbb{E}\left[ \Delta Y_{\frac jn}\Delta Y_{\frac kn}\right] \right| \le Cn^{-\alpha}(j+k)^{\alpha-2}\] for a constant $C >0$ that is independent of $j$, $k$ and $n$.

\end{enumerate} \end{lemma}
\begin{proof}
Property (a) is well-known for fractional Brownian motion. For (b), we have from \req{CovYY}:
\begin{align*}
\mathbb{E}\left[ \Delta Y_{\frac jn} \Delta Y_{\frac kn}\right] &=\frac{2c_2}{\Gamma(\gamma)C_H^2 n^\alpha} \int_0^\infty   y^{-\alpha-1} \left( e^{-yj}-e^{-y(j+1)}\right)\left( e^{-yk}-e^{-y(k+1)}\right)dy\\
&= \frac{2c_2}{\Gamma(\gamma)C_H^2 n^\alpha}  \int_0^\infty y^{-\alpha+1}\int_0^1\int_0^1 e^{-y(j+k+u+v)}du~dv~dy.
\end{align*}Note that the above integral is nonnegative, and we can bound this with
\begin{align*}
\left|\mathbb{E}\left[ \Delta Y_{\frac jn} \Delta Y_{\frac kn}\right]\right| &\le Cn^{-\alpha}\int_0^\infty y^{-\alpha+1}e^{-y(j+k)}~dy\\
&= Cn^{-\alpha}(j+k)^{\alpha-2}\int_0^\infty u^{-\alpha+1}e^{-u}du\\
&\le Cn^{-\alpha}(j+k)^{\alpha-2}.
\end{align*}
\end{proof}
 
\begin{lemma}\label{Lem4_2}   For $n\ge 2$ fixed and integers $j,k \ge 1$,
\[\left| \mathbb{E}\left[ \Delta W_{\frac jn}\Delta Y_{\frac kn}\right] \right| \le Cn^{-\alpha}j^{2H-2} k^{-\gamma}
\] for a constant $C >0$ that is independent of $j$, $k$ and $n$.   
\end{lemma}

\begin{proof}
From \req{UVY} - (3.6) in the proof of Theorem \ref{thm1}, observe that
\[ \mathbb{E}\left[ \Delta W_{\frac jn} \Delta Y_{\frac kn}\right] = \mathbb{E}\left[ (\Delta U_{\frac jn}+\Delta V_{\frac jn})\Delta Y_{\frac kn}\right] = \mathbb{E}\left[ \Delta U_{\frac jn}\Delta Y_{\frac kn}\right].\]
Assume $s, t >0$.  By self-similarity we can define the covariance function $\psi$  by $\mathbb{E}\left[ U_t Y_s\right] = s^\alpha\mathbb{E}\left[ U_{t/s}Y_1\right] = s^\alpha\psi(t/s)$, where, using the change-of-variable $\theta = \eta x$,
\begin{align*}
\psi(x)&= \int_0^\infty \int_0^\infty y^{-\alpha-1}\frac{\eta^{1-2H}}{1+\eta^2}\left(\cos(y \eta x)-1\right)(1-e^{-y})~d\eta~dy\\
&=\int_0^\infty y^{-\alpha-1}(1-e^{-y}) \int_0^\infty \frac{\theta^{1-2H}x^{2H}}{x^2+\theta^2}\left(\cos(y \theta)-1\right)~d\theta~dy.
\end{align*} 
Then using the fact that
\beq{xtheta} \left|\frac{\theta^{1-2H}x^{2H}}{x^2+\theta^2}\right| \le |\theta^{-2H}|~|x|^{2H-1},\eeq
we see that $|\psi(x)| \le C x^{2H -1}$, and 
\begin{align*}
\psi'(x) &= 2H\int_0^\infty y^{-\alpha-1}(1-e^{-y}) \int_0^\infty \frac{\theta^{1-2H}x^{2H-1}}{x^2+\theta^2}\left(\cos(y \theta)-1\right)~d\theta~dy\\
&\qquad - 2\int_0^\infty y^{-\alpha-1}(1-e^{-y}) \int_0^\infty \frac{\theta^{1-2H}x^{2H+1}}{(x^2+\theta^2)^2}\left(\cos(y \theta)-1\right)~d\theta~dy.
\end{align*}
Using \req{xtheta} and similarly
\beq{xtheta2}\left|\frac{\theta^{1-2H}x^{2H+1}}{(x^2+\theta^2)^2}\right| \le |\theta^{-2H}|~|x|^{2H-2},\eeq
we can write
\[ \left| \psi'(x)\right| \le x^{2H -2}|2H -2| \int_0^\infty y^{-\alpha -1}(1-e^{-y})\int_0^\infty \theta^{-2H}\left(\cos(y \theta) - 1\right)~d\theta~dy\le Cx^{2H -2}.\]
By continuing the computation, we can find that $| \psi''(x)| \le Cx^{2H-3}$.
We have for $j,k\ge 1$,
\begin{align*}
\mathbb{E}\left[ \Delta U_{\frac jn}\Delta Y_{\frac kn}\right] &= n^{-\alpha}(k+1)^\alpha\left( \psi\left(\frac{j+1}{k+1}\right)-\psi\left(\frac{j}{k+1}\right)\right)\\
&\qquad-n^{-\alpha}k^{\alpha}\left( \psi\left(\frac{j+1}{k}\right) - \psi\left(\frac jk\right)\right)\\
&= n^{-\alpha}\left( (k+1)^\alpha-k^{\alpha}\right)\left( \psi\left(\frac{j+1}{k+1}\right)-\psi\left(\frac{j}{k+1}\right)\right)\\
&\quad + n^{-\alpha}k^\alpha\left( \psi\left(\frac{j+1}{k+1}\right)-\psi\left(\frac{j}{k+1}\right)-\psi\left(\frac{j+1}{k}\right) + \psi\left(\frac jk\right)\right).
\end{align*}
With the above bounds on $\psi$ and its derivatives, the first term is bounded by
\begin{multline*}
n^{-\alpha}\left| (k+1)^\alpha-k^{\alpha}\right|\left| \psi\left(\frac{j+1}{k+1}\right)-\psi\left(\frac{j}{k+1}\right)\right|\\ \le \alpha n^{-\alpha}\int_0^1 (k+u)^{\alpha-1}du \int_0^{\frac{1}{k+1}}\left|\psi'\left( \frac{j}{k+1}+v\right)\right|~dv\\
\le Cn^{-\alpha}k^{\alpha-2}\left(\frac jk\right)^{2H-2} \le Cn^{-\alpha}k^{-\gamma}j^{2H -2},
\end{multline*}
and
\begin{multline*}
n^{-\alpha}k^\alpha\left| \psi\left(\frac{j+1}{k+1}\right)-\psi\left(\frac{j}{k+1}\right)-\psi\left(\frac{j+1}{k}\right) + \psi\left(\frac jk\right)\right|\\
=n^{-\alpha}k^\alpha\left| \int_0^{\frac{1}{k+1}} \psi'\left(\frac{j}{k+1}+u\right)~du - \int_0^{\frac{1}{k}} \psi'\left(\frac{j}{k}+u\right)~du\right|\\
\le n^{-\alpha}k^\alpha \int_{\frac{1}{k+1}}^{\frac{1}{k}}\left| \psi'\left( \frac jk +u\right)\right|~du +\int_0^{\frac{1}{k+1}}\int_{\frac{j}{k+1}}^{\frac{j}{k}}\left|\psi''(u+v)\right|~dv~du\\
\le Cn^{-\alpha}k^{\alpha-2}\left(\frac jk\right)^{2H-2} +Cn^{-\alpha}k^{\alpha-3}j\left(\frac jk\right)^{2H -3} 
\le Cn^{-\alpha}k^{-\gamma}j^{2H-2}.
\end{multline*} 
This concludes the proof of the lemma. 
\end{proof}

\subsection{Proof of Theorem \ref{thm2}}
We will make use of the notation $\beta_{j,n}= \left\| \DXj\right\|_{L^2(\Omega)}$.
From Lemma  \ref{Lem4_1} and Lemma \ref{Lem4_2} we have
\[
\beta_{j,n}^2 =\kappa n^{-\alpha} (1+ \theta_{j,n}),
\]
where $|\theta_{j,n} | \le C j^{\alpha -2}$ if $j\ge 1$. Notice that, in the definition of $F_n(t)$, it suffices to consider the sum for $j\ge n_0$ for a fixed  $n_0$. Then,   we can choose $n_0$ in such a way that $   C n_0 ^{\alpha -2}  \le \frac 12  $, which implies
\begin{equation}  \label{beta}
\beta_{j,n}^2 \ge \kappa n^{-\alpha} (1-C j^{\alpha -2})
\end{equation}
 for any $j \ge n_0$.

By \req{Hmap},
\[ 
 \beta_{j,n}^q H_q \left(  \beta_{j,n} ^{-1} \DXj \right) = I^X_q\left( \left(\mathbf{1}_{[\frac jn,\frac{j+1}{n})} \right)^{\otimes q}\right),
\]  
where $I^X_q$ denotes the multiple stochastic integral of order $q$ with respect to the process $X$.
Thus, we can write  
\[ 
F_n(t) = n^{-\frac 12}\sum_{j=n_0}^{\Nt -1}  \beta_{j,n } ^{-q}  I^X_q \left(\mathbf{1}_{\left[ \frac jn, \frac{j+1}{n}\right)}^{\otimes q} \right).
\]
The decomposition $X=W+Y$ leads to
\[
I^X_q \left(\mathbf{1}_{\left[ \frac jn, \frac{j+1}{n}\right)}^{\otimes q} \right)=
\sum_{r=0}^q \binom qr  I^W_r \left(\mathbf{1}_{\left[ \frac jn, \frac{j+1}{n}\right)}^{\otimes r} \right) I^Y_{q-r} \left(\mathbf{1}_{\left[ \frac jn, \frac{j+1}{n}\right)}^{\otimes {q-r}} \right).
\]
We are going to show that the terms with $r=0,\dots, q-1$ do not contribute to the limit. 
Define
\[ G_n(t) = n^{-\frac{1}{2}}\sum_{j=n_0}^{\Nt -1} \beta_{j,n}^{-q}   I^W_q \left(\mathbf{1}_{\left[ \frac jn, \frac{j+1}{n}\right)}^{\otimes q} \right)\]
and
\[ \widetilde{G}_n(t) = n^{-\frac{1}{2}}\sum_{j=n_0}^{\Nt -1} \left\| \Delta W_{j/n} \right\|^{-q}_{L^2(\Omega)}  I^W_q \left(\mathbf{1}_{\left[ \frac jn, \frac{j+1}{n}\right)}^{\otimes q} \right).
\]
Consider the decomposition
\[
F_n(t)= (F_n(t)- G_n(t)) + (G_n(t) - \widetilde{G}_n(t) ) + \widetilde{G}_n(t).
\]
Notice that all these processes vanish at $t=0$. 
We claim that for any $0\le s<t \le T$, we have
\begin{equation} \label{4.4}
 \mathbb{E} [ |F_n(t)- G_n(t)- (F_n(s) -G_n(s)) |^2 ] \le    \frac {    (  \Nt -\Ns  )^\delta} n
\end{equation}
and
\begin{equation} \label{4.5}
 \mathbb{E} [ |G_n(t)- \widetilde{G}_n(t)- (G_n(s) -\widetilde{G}_n(s)) |^2 ] \le   \frac {    (  \Nt -\Ns  )^\delta} n,
\end{equation}
where  $0\le \delta <1$.
By Lemma \ref{Lem4_1}, $\left\| \Delta W_{j/n} \right\|^2_{L^2(\Omega)} = \kappa n^{-\alpha}$ for every $j$.  As a consequence, using  \req{Hmap} we can also write
\[ \widetilde{G}_n(t) = n^{-\frac 12}\sum_{j=n_0}^{\Nt -1} H_q \left(  \kappa^{-\frac 12} n^{\frac{\alpha }{2}} \Delta W_{\frac jn}  \right).\]
Since $\kappa ^{-\frac 12} W$ is a fractional Brownian motion, the Breuer-Major theorem implies  that  the process $\widetilde{G} $ converges in $D([0,T])$ to a scaled Brownian motion $\{ \sigma B_t, t\in [0,T]\}$,  where $\sigma^2$ is given in (\ref{sigma}). By the fact that all the $p$-norms are equivalent on a fixed Wiener chaos, the estimates (\ref{4.4}) and (\ref{4.5}) lead to
\begin{equation} \label{4.6}
 \mathbb{E} [ |F_n(t)- G_n(t)- (F_n(s) -G_n(s)) |^{2p} ] \le \frac {    (  \Nt -\Ns  )^{\delta p}} {n^p}
 \end{equation}
and
\begin{equation} \label{4.7}
 \mathbb{E} [ |G_n(t)- \widetilde{G}_n(t)- (G_n(s) -\widetilde{G}_n(s)) |^{2p} ] \le   \frac {    (  \Nt -\Ns  )^{\delta p}} {n^p},
\end{equation}
for all $p\ge 1$. 
Letting $n$ tend to infinity,   we deduce from (\ref{4.6}) and  (\ref{4.7})   that  the sequences
$F_n -G_n$  and  $G_n -\widetilde{G}_n$ converge  to zero in the topology of  $D([0,T])$, as $n$ tends to infinity.

\medskip
\noindent
{\it Proof of (\ref{4.4}):}  
We can write
\[
\mathbb{E} \left[ |F_n(t) -G_n(t)- (F_n(s) -G_n(s))|^2 \right] \le  C \sum_{r=0}^{q-1}   \mathbb{E} [\Phi_{r,n}^2] ,
\]
where
\[
 \Phi_{r,n}=n^{-\frac 12  }\sum_{j=\Ns \vee n_0}^{\Nt -1}  \beta_{j,n} ^{-q}  I^W_r \left(\mathbf{1}_{\left[ \frac jn, \frac{j+1}{n}\right)}^{\otimes r} \right) I^Y_{q-r} \left(\mathbf{1}_{\left[ \frac jn, \frac{j+1}{n}\right)}^{\otimes {q-r}} \right).
 \]
 We have, using (\ref{beta}),
 \begin{eqnarray*}
 && \mathbb{E} [\Phi_{r,n}^2] \le  n^{-1+q\alpha}  \\
&& \quad \times  \sum_{j,k=\Ns \vee n_0} ^{\Nt -1}  \left|  \mathbb{E} \left[    I^W_r \left(\mathbf{1}_{\left[ \frac jn, \frac{j+1}{n}\right)}^{\otimes r} \right) I^Y_{q-r} \left(\mathbf{1}_{\left[ \frac jn, \frac{j+1}{n}\right)}^{\otimes {q-r}} \right)
  I^W_r \left(\mathbf{1}_{\left[ \frac kn, \frac{k+1}{n}\right)}^{\otimes r} \right) I^Y_{q-r} \left(\mathbf{1}_{\left[ \frac kn, \frac{k+1}{n}\right)}^{\otimes {q-r}} \right) \right] \right|.
  \end{eqnarray*}
Using a diagram method for the expectation of four stochastic integrals (see \cite{PTaq}), we find that, for any $j,k$,  the above expectation consists of a sum of terms of the form
\[
 \left(\mathbb{E}\left[ \Delta W_{\frac jn}\Delta W_{\frac kn}\right]\right)^{a_1}
\left(\mathbb{E}\left[ \Delta Y_{\frac jn}\Delta Y_{\frac kn}\right]\right)^{a_2}
\left(\mathbb{E}\left[ \Delta W_{\frac jn}\Delta Y_{\frac kn}\right]\right)^{a_3}
\left(\mathbb{E}\left[ \Delta Y_{\frac jn}\Delta W_{\frac kn}\right]\right)^{a_4},
\]
where  the $a_i$ are nonnegative integers such that $a_1 +a_2 + a_3 + a_4 = q$, $a_1 \le r \le q-1$, and $a_2 \le q-r$.  First, consider the case with $a_3 = a_4 = 0$, so that we have the sum
\[
n^{-1+q\alpha}\sum_{j,k=\Ns \vee n_0}^{\Nt -1} \left(\mathbb{E}\left[ \Delta W_{\frac jn}\Delta W_{\frac kn}\right]\right)^{a_1}
\left(\mathbb{E}\left[ \Delta Y_{\frac jn}\Delta Y_{\frac kn}\right]\right)^{q-a_1},
\]
where $0 \le a_1  \le q-1$.  Applying Lemma \ref{Lem4_1}, we can control each of the terms  in the above sum by
\[
n^{-q\alpha }(|j-k+1|^\alpha-2|j-k|^\alpha+|j-k-1|^\alpha)^{a_1} (j+k)^{(q-a_1)(\alpha-2)},
\]
which gives
\begin{eqnarray}  \notag
&& n^{-1+q\alpha}\sum_{j,k=\Ns \vee n_0}^{\Nt -1} \left|\mathbb{E}\left[ \Delta W_{\frac jn}\Delta W_{\frac kn}\right]\right|^{a_1}
\left|\mathbb{E}\left[ \Delta Y_{\frac jn}\Delta Y_{\frac kn}\right]\right|^{q-a_1}\\ \notag
&& \qquad\le  Cn^{-1} \left( \sum_{j=\Ns \vee n_0}^{\Nt -1} j^{\alpha -2} +
\sum_{j,k=\Ns \vee n_0, j\not =k}^{\Nt -1}  |j-k| ^{(q-1)(\alpha-2)}(j+k)^{\alpha-2}  \right)\\  \notag
&& \qquad \le   Cn^{-1}  \sum_{j=\Ns \vee n_0}^{\Nt -1}  \left( j^{\alpha -2} + j^{ q(\alpha -2)+1} \right) \\
&& \qquad \le    Cn^{-1} \left(\Nt -\Ns) ^{(\alpha -1) \vee 0} + (\Nt -\Ns) ^{[q(\alpha -2)+ 2]\vee 0} \right).
  \label{tr1}
\end{eqnarray}

Next, we consider the case where $a_3 + a_4 \ge 1$.  By Lemma \ref{Lem4_1}, we have that, up to a constant $C$,
\[ \left|\mathbb{E}\left[ \Delta Y_{\frac jn}\Delta Y_{\frac kn}\right]\right|\le C \left|\mathbb{E}\left[ \Delta W_{\frac jn}\Delta W_{\frac kn}\right]\right|,\]
so we may assume $a_2 =0$, and  have to handle the term 
\begin{equation}\label{UY_YU_AB} n^{-1+q\alpha}\sum_{j,k=\Ns \vee n_0}^{\Nt -1} \left|\mathbb{E}\left[ \Delta W_{\frac jn}\Delta W_{\frac kn}\right]\right|^{q-a_3-a_4}\left|\mathbb{E}\left[ \Delta W_{\frac jn}\Delta Y_{\frac kn}\right]\right|^{a_3}\left|\mathbb{E}\left[ \Delta Y_{\frac jn}\Delta W_{\frac kn}\right]\right|^{a_4} 
\end{equation}
for all allowable values of $a_3, a_4$ with $a_3 + a_4 \ge 1$.
Consider the decomposition
\begin{align*}
n^{-1+q\alpha}&\sum_{j,k=\Ns \vee n_0}^{\Nt -1} \left|\mathbb{E}\left[ \Delta W_{\frac jn}\Delta W_{\frac kn}\right]\right|^{q-a_3-a_4}\left|\mathbb{E}\left[ \Delta W_{\frac jn}\Delta Y_{\frac kn}\right]\right|^{a_3}\left|\mathbb{E}\left[ \Delta Y_{\frac jn}\Delta W_{\frac kn}\right]\right|^{a_4}\\
&\;=n^{q\alpha-1}\sum_{j=\Ns \vee n_0}^{\Nt -1}\left|\mathbb{E}\left[ \Delta W_{\frac jn}^2\right]\right|^{q-a_3-a_4}~\left|\mathbb{E}\left[ \Delta W_{\frac jn}\Delta Y_{\frac jn}\right]\right|^{a_3+a_4}\\
&\;+n^{q\alpha-1}\sum_{j=\Ns \vee n_0}^{\Nt -1}\sum_{k=\Ns \vee n_0}^{j-1}\left|\mathbb{E}\left[ \Delta W_{\frac jn}\Delta W_{\frac kn}\right]\right|^{q-a_3-a_4}\left|\mathbb{E}\left[ \Delta W_{\frac jn}\Delta Y_{\frac kn}\right]\right|^{a_3}\left|\mathbb{E}\left[ \Delta Y_{\frac jn}\Delta W_{\frac kn}\right]\right|^{a_4}\\
&\;+n^{q\alpha-1}\sum_{k=\Ns \vee n_0}^{\Nt -1}\sum_{j=\Ns \vee n_0}^{k-1}\left|\mathbb{E}\left[ \Delta W_{\frac jn}\Delta W_{\frac kn}\right]\right|^{q-a_3-a_4}\left|\mathbb{E}\left[ \Delta W_{\frac jn}\Delta Y_{\frac kn}\right]\right|^{a_3}\left|\mathbb{E}\left[ \Delta Y_{\frac jn}\Delta W_{\frac kn}\right]\right|^{a_4}.
\end{align*}
We have,by Lemma \ref{Lem4_1} and Lemma \ref{Lem4_2},
\begin{align}  \notag 
n^{-1+q\alpha}&\sum_{j,k=\Ns \vee n_0}^{\Nt -1} \left|\mathbb{E}\left[ \Delta W_{\frac jn}\Delta W_{\frac kn}\right]\right|^{q-a_3-a_4}\left|\mathbb{E}\left[ \Delta W_{\frac jn}\Delta Y_{\frac kn}\right]\right|^{a_3}\left|\mathbb{E}\left[ \Delta Y_{\frac jn}\Delta W_{\frac kn}\right]\right|^{a_4}\\  \notag
&\quad \le Cn^{-1}\sum_{j=\Ns \vee n_0}^{\Nt -1}j^{(a_3+a_4)(\alpha-2)}\\ \notag
&\qquad + Cn^{-1}\sum_{j=\Ns \vee n_0}^{\Nt -1} j^{a_3(2H-2)-a_4\gamma}\sum_{k=\Ns \vee n_0}^{j-1}k^{-a_3\gamma+a_4(2H -2)}\left| j-k\right|^{(q-a_3-a_4)(\alpha-2)}\\\notag
&\qquad + Cn^{-1}\sum_{k=\Ns \vee n_0}^{\Nt -1} k^{-a_3\gamma+a_4(2H-2)}\sum_{j=\Ns \vee n_0}^{k-1}j^{a_3(2H-2)-a_4\gamma}\left| k-j\right|^{(q-a_3-a_4)(\alpha-2)}\\ \notag
&\quad    \le Cn^{-1}  \left(  (\Nt-\Ns) ^{[(a_3+a_4)( \alpha -2)+1]\vee 0}+  (\Nt-\Ns) ^{[q(\alpha -2)+2]\vee 0}  \right. \\ \label{tr2}
 &\qquad  \left.+ (\Nt-\Ns) ^{[a_3(2H-2)-a_4\gamma +1]\vee 0}+  (\Nt-\Ns) ^{[a_4(2H-2)-a_3\gamma +1] \vee 0} \right).
\end{align}
Then (\ref{tr1}) and (\ref{tr2}) imply  (\ref{4.4}) because  $\alpha <2-\frac 1q$.

 \medskip
 \noindent
 {\it Proof of (\ref{4.5}):}
 We have
\[ 
G_n(t) - \widetilde{G}_n(t) = n^{-\frac 12}\sum_{j=n_0}^{\Nt -1}\left(  \beta_{j,n}^{-q}  - \left\| \Delta W_{\frac jn} \right\|^{-q}_{L^2(\Omega)}\right) I^W_q\left( \mathbf{1}_{[\frac jn, \frac{j+1}{n})}^{\otimes q}\right)
\]
and we can write, using  (\ref{beta}) for any $j\ge n_0$,
\[
 \left| \beta_{j,n}^{-q}  - \left\| \Delta W_{\frac jn} \right\|^{-q}_{L^2(\Omega)} \right|
 = (\kappa^{-1} n^{\alpha})^{\frac q2} \left|(1+ \theta _{j,n})^{-\frac q2} -1 \right| \le   C\left(\kappa^{-1} n^{\alpha } j^{\alpha -2}\right)^{\frac q2} .
 \]
This leads to the estimate
\begin{eqnarray*}
&&  \mathbb{E}\left[\left| G_n(t) - \widetilde{G}_n(t)- (G_n(s) - \widetilde{G}_n(s))  \right|^2\right]
 \le  C   n^{-1}  \\
&& \qquad  \times  \left(  \sum_{j=\Ns \vee n_0}^{\Nt -1} j^{\alpha -2} +
\sum_{j,k=\Ns \vee n_0, j\not =k}^{\Nt -1}  |j-k| ^{q (\alpha-2)}  \right)\\
 && \le   C   n^{-1} \left(\Nt -\Ns) ^{(\alpha -1) \vee 0} + (\Nt -\Ns) ^{[q(\alpha -2)+ 2]\vee 0} \right)  ,
\end{eqnarray*}
which implies (\ref{4.5}).

This concludes the proof of Theorem \ref{thm2}.

\end{document}